\documentclass[12pt]{amsart}

\usepackage{amsthm,amssymb}
\usepackage{fullpage}
\usepackage[utf8]{inputenc}
\usepackage{graphicx} 
\usepackage{subcaption}
\usepackage{amssymb}
\usepackage{xcolor}
\usepackage{gensymb}
\usepackage{float}
\usepackage{soul,xcolor}
\usepackage[utf8]{inputenc}
\usepackage[T1]{fontenc}
\usepackage[english]{babel}
\usepackage{enumitem}

\newcommand{\embed}{\mathsf{L}}
\newcommand{\sign}{\mathsf{sign}}
\newcommand{\dett}{\mathsf{det}}
\newcommand{\support}{\mathsf{supp}}

\newcommand{\f}[2]{\mathbb{F}_{#1}^{#2}}
\newcommand{\w}[1]{\mathbf{#1}}

\usepackage{soul,xcolor}

\newcommand{\stw}{\mathbb{S}^2}
\newcommand{\sth}{\mathbb{S}^3}

\newcommand{\rth}{\mathbb{R}^3}

\newcommand{\fefe}{\mathbb{F}}
\newcommand{\rere}{\mathbb{R}}

\newtheorem{proposition}{Proposition}
\newtheorem{theorem}{Theorem}
\newtheorem{lemma}{Lemma}

\newtheorem{question}{Question}

\newtheorem{remark}{Remark}
\newtheorem{claim}{Claim}

\newtheorem{example}{Example}

\begin{document}
\setstcolor{red}

\title[Computing the determinant of links through Fourier-Hadamard transforms]{Computing the determinant of links through Fourier-Hadamard transforms} 

\thanks{$^1$ Partially supported by grant PRIM80-CNRS}
\author[Baptiste Gros]{Baptiste Gros$^1$}
\address{IMAG, Univ.\ Montpellier, CNRS, Montpellier, France}
\email{baptiste.gros@umontpellier.fr}
\author[Ulises Pastor-D\'iaz]{Ulises Pastor-D\'iaz}
\address{Departamento de \'Algebra, Facultad de Matem\'aticas, Universidad de Sevilla, Spain}
\email{upastor@us.es}
\author[Jorge L. Ram\'irez Alfons\'in]{Jorge L. Ram\'irez Alfons\'in}
\address{IMAG, Univ.\ Montpellier, CNRS, Montpellier, France}
\email{jorge.ramirez-alfonsin@umontpellier.fr}



\begin{abstract} In this paper, we present a  novel method to compute the determinant of a link using Fourier-Hadamard transforms of Boolean functions. 
We also investigate the determinant of {\em centrally symmetric} links (a special class of {\em strong achiral} links). In particular, we show that the determinant of a centrally symmetric link with an even number of components is equals zero. 
\end{abstract}

\maketitle

\section{Introduction}\label{sec;intro}

In this paper, we investigate the {\em determinant} of a link $L$, denoted by $\dett(L)$. The determinant is a useful invariant appearing in different contexts. For instance, when $L$ is a knot then $\dett(L)$ is the order of the first homology of the two-fold cover of $\sth$ branched over $L$ \cite{Rolf}. It si also strongly connected with diophantine equations \cite{St} and to the study of hyperbolic volumes of (complements in $\mathbb{S}^3$ of) alternating knots and links \cite{St1}.
\smallskip

There are many different ways to compute $\dett(L)$. As the determinant of either its {\em Seifert} matrix \cite[page 213]{Rolf}, or its {\em Goeritz} matrix \cite{GL} or its {\em coloring} matrix \cite[Chapter 3]{Liv}, as the evaluation of either its {\em Jones} polynomial $|V_L(-1)|$ or its {\em Alexander} polynomial $|\Delta_L(-1)|$  \cite[(12.3)]{jones} or its {\em Kauffman} bracket $|\langle D\rangle_{A=e^{\pi i/4}}|$ where $D$ is some diagram of $L$ \cite{Krebes}. The latter is explained in detail in Appendix \ref{Append}. More recently,  Dasbach {\em et al.} \cite[Theorem 3.2]{DFK} presented a formula to compute the determinant given by an alternating sum of the number of {\em quasi-trees of genus} $j$ of a {\em dessin d'enfant} $H$ (combinatorial map, oriented ribbon graph) consisting of $g(H)$ terms, where $g(H)$ is the genus of the surface where the dessin $H$ is embedded. 
\medskip


An {\em edge-signed} planar graph, denoted by $(G,\chi_E)$, is a planar graph $G$ equipped with a signature on its edges $\chi_E:E\rightarrow \{+,-\}$. 
Let $D_L$ be a diagram of a link $L$. It is known (see Section \ref{sec:back}) that it can be associated to $D_L$ a unique edge-signed planar graph $(G,\chi_E)$ from which $D_L$ can be uniquely determined ($G$ is called the {\em Tait graph} associated to $D_L$). Let $T$ be a spanning tree of $G$. We define
$$\sign(T)=\prod\limits_{e\in E(T)} \chi(e)$$
where $\chi(e)$ denotes the sign of edge $e$.
\smallskip

We say that $T$ is {\em positive} (resp. {\em negative}) if $\sign(T)=+$ (resp. $\sign(T)=-$).
\smallskip

The following less known simple combinatorial approach to compute $\dett(L)$ in terms of signed-spanning trees solely was proposed in \cite[Lemma 2.2]{CK}. 

\begin{lemma}\label{lem:main}\cite[Lemma 2.2]{CK}  Let $L$ be the link arising from the edge-signed connected planar graph $(G,\chi_E)$. Then, 
$$\dett(L)= \big| \# \{\text{positive-spanning trees of } G\}-\# \{\text{negative-spanning trees of } G\}\big|.$$
\end{lemma}

 \begin{remark}
(a) If $L$ is {\em alternating}, then all the edges have the same sign and thus, in this case, all the spanning trees are either positive or negative. Therefore, the difference between positive and negative spanning trees would be simply the number of spanning trees of $G$. The latter is a well-known result that can be traced back at least as far as Crowell \cite{Crowell}. 
\smallskip

(b) A similar expression for $\dett(L)$ is given in \cite[Corollary 4.2]{DFK} but in terms of the dessin $H$ and its dual, where $H$ is assumed to be of {\em Turaev genus} one. The latter allows to compute $\dett(L)$ only in the case when $L$ is of Turaev genus one.  
\end{remark}

Champanerkar and Kofman \cite{CK} obtained Lemma \ref{lem:main} by considering a specific evaluation of the {\em spanning tree expansion} of the Jones polynomial, introduced by Thistlethwaite \cite{Th}, and by using certain gradings introduced by the same authors in \cite{CK1}. In the Appendix, we propose a less artificial alternative proof of Lemma \ref{lem:main}. Our approach, based on the state model for the Kauffman bracket, shed light on the role of signed-spanning trees in such models and thus giving further understanding of the connection with the determinant.
\smallskip

\subsection{Fourier-Hadamard transforms}
Our first goal is to introduce a novel method to compute $\dett(L)$ for any link $L$. We do so by using Fourier-Hadamard transforms of Boolean fonctions. Let $f : \fefe_2^n\rightarrow \fefe_2$ be a Boolean function and let $\support(f)=\{ {\bf x}\in \fefe_2^n \ \vert \  f({\bf x})\neq 0\}$ be its {\em support}. The {\em Fourier-Hadamard transform} of $f$ is defined as
$$\widehat f({\bf u})=\sum\limits_{{\bf x}\in \fefe_2^n} f({\bf x}) (-1)^{{\bf x}\cdot {\bf u}}=\sum\limits_{{\bf x}\in \support(f)} (-1)^{{\bf x}\cdot {\bf u}}$$

where ``$\cdot$'' is some chosen inner product in $ \fefe_2^n$ (for instance the usual inner product ${\bf x}\cdot {\bf u}=x_1u_1\oplus \cdots \oplus x_n u_n$ where $\oplus$ denotes the sum modulo 2).
\smallskip

The Fourier-Hadamard {\em spectrum} of $f$ is the multiset of all possible values of $\widehat f({\bf u})$.
\smallskip

It is known \cite{Ca} that any Boolean function $f$ (in fact any  {\em pseudo-Boolean} function) can be represented by elements in the quotient ring $\mathbb{R}[x_1,\ldots,x_n]/(x_1^2-x_1,\ldots,x_n^2-x_n)$. We shall call it the {\em Numerical Normal Form} (NNF). Such a representation can be written as 
$$f({\bf x})=\sum\limits_{{\bf y}\in \fefe_2^n} \lambda_{\bf y} {\bf x}^{\bf y}$$
where ${\bf x}^{\bf y}=\prod\limits_{i=1}^n x_i^{y_i}$.
\smallskip

We refer the reader to \cite{Ca} for further background on the Fourier-Hadamard transform. 
\smallskip

Let $G=(V,E)$ be a connected planar graph with $n=|E|$ and let $\mathcal{T}_G$ be the set of spanning trees of $G$. For any $F\subset E$, we let ${\bf v}_F=(v_1,\dots ,v_n)$ be the {\em characteristic vector} of $F$, that is, $v_i=1$ if $i\in F$ and 0 otherwise. Let $f_G$ be the Boolean function with $\support(f_G)=\{ {\bf v}_T \in \fefe_2^n\  \vert \ T\in \mathcal{T}_G \}$. We define the {\em $FH_G$-polynomial} of $G$, denoted by $FH_G(x_1,\ldots , x_n)$, to be the NNF of the Fourier-Hadamard transform of $f_G$, that is, 
$$FH_G(x_1,\ldots , x_n)=\widehat f_G(x_1,\ldots , x_n).$$ 

\begin{theorem}\label{theo:main}  Let $L_G$ be the link arising from the edge-signed connected planar graph $(G,\chi_E)$. Then,
$$\dett(L_G)=\big| FH_G({\bf v})\big|$$

where ${\bf v}=(v_1,\ldots ,v_n)$ with $v_i=\frac{1-\chi_E(i)}{2}$.
\end{theorem}

Here is an interesting feature of our approach. Consider the following

\begin{question} Let $k\ge 0$ be an integer and let $G=(E,V)$ be a planar connected graph. Is there an edge-signature $\chi_E$ such that $\dett(L)=k$ where $L$ is the link arising from $(G,\chi_E)$ ? 
\end{question}  

Our method yields to a straightforward procedure to answer the above question. It is enough to calculate the $FH_G$-polynomial (that is, $FH_G(\w{x})=\widehat f_G(x_1,\ldots , x_n)$) and then consider the spectrum of $\widehat  f_G$. Indeed, it turns out that, by construction, each value in the spectrum of $\widehat  f_G$ corresponds to an evaluation of $FH_G(\w{x})$ on a specific vector arising from a signature $\chi'_E$ of $E(G)$. Therefore, by Theorem \ref{theo:main}, the set of the values in the spectrum of $\widehat  f_G$ corresponds to the set of values given by $\dett(L)$ where $L$ is the link arising from $(G,\chi'_E)$. Two ways to calculate $FH_G(\w{x})$ (one an explicit formula and the other a recursive method) will be discussed.
\medskip

\subsection{Centrally symmetric links}
Our second goal is to investigate the determinant of {\em centrally symmetric} links. Let
\[\begin{array}{lclc}
c: & \mathbb{R}^3& \rightarrow & \mathbb{R}^3\\ 
& (x,y,z) & \mapsto & (-x,-y,-z)
\end{array}
\]
be the {\em central symmetric} function.
 
We say that a link $L$ is {\em centrally symmetric} if it admits a representation $\hat L$ in $ \mathbb{R}^3$ such that $c(\hat L)=\hat L$. 

\begin{remark} Recall that a link is {\em strongly achiral} if it admits a representation in $\mathbb{R}^3$ which is preserved by some orientation-reversing diffeomorphism. We notice that such diffeomorphisms can be conjugated to either $c$ or the function $i: (x,y,z)\mapsto (-x,y,z)$.  We thus have that centrally symmetric links are a special case of strong achirality.
\end{remark}

By using Lemma \ref{lem:main}, we are able to show the following

\begin{theorem}\label{th:det:sym} Let $L$ be a centrally symmetric link with $n$ components. If $n$ is even then $\dett(L)=0$.
\end{theorem}

The paper is organized as follows. In the next section, we review some necessary background in knot theory as well as some needed notions on Fourier-Hadamard transform of Boolean functions. In Section \ref{sec:FH}, we prove Theorem \ref{theo:main} and discuss two methods to compute the $FH_G$-polynomial (Propositions \ref{baseform} and \ref{recformula}). In Section \ref{sec:symmetry}, we show Theorem \ref{th:det:sym}. On the way to prove the latter, we also present a combinatorial characterization for a link to be centrally symmetric (Theorem \ref{thm:centralsymm}) as well as  a characterization for the parity on the number of its components (Lemma \ref{lem:parity}). Finally, in Appendix \ref{Append}, we present a proof of Lemma \ref{lem:main}.

\section{Preliminaries}\label{sec:back}

\subsection{Knot theory}\label{sec;prem} We refer the reader to \cite{Adam} for standard background on knot theory. 

A {\em link diagram} $D(L)$ of a link $L$ is a regular projection of $L$ into $\mathbb{R}^2$ in such a way that the projection of each component is smooth, no self-intersected  and at most two curves intersect at any point. At such a crossing, the curves intersect transversally, and a diagram has a finite number of intersection points. At each crossing point of the link diagram the curve which goes over the other is specified. The {\em shadow} of a link diagram $D$ is the 4-regular graph obtained when the over/under passes of $D$ are ignored. We say that the diagram $D$ is {\em connected} if the shadow is connected. Since the shadow is Eulerian (4-regular) then its faces can be 2-colored, say with colors black and white. We thus have that each vertex is incident to 4 faces alternatively colored around the vertex; see Figure \ref{fig1}.

\begin{figure}[H]
\centering
\includegraphics[width=.54\linewidth]{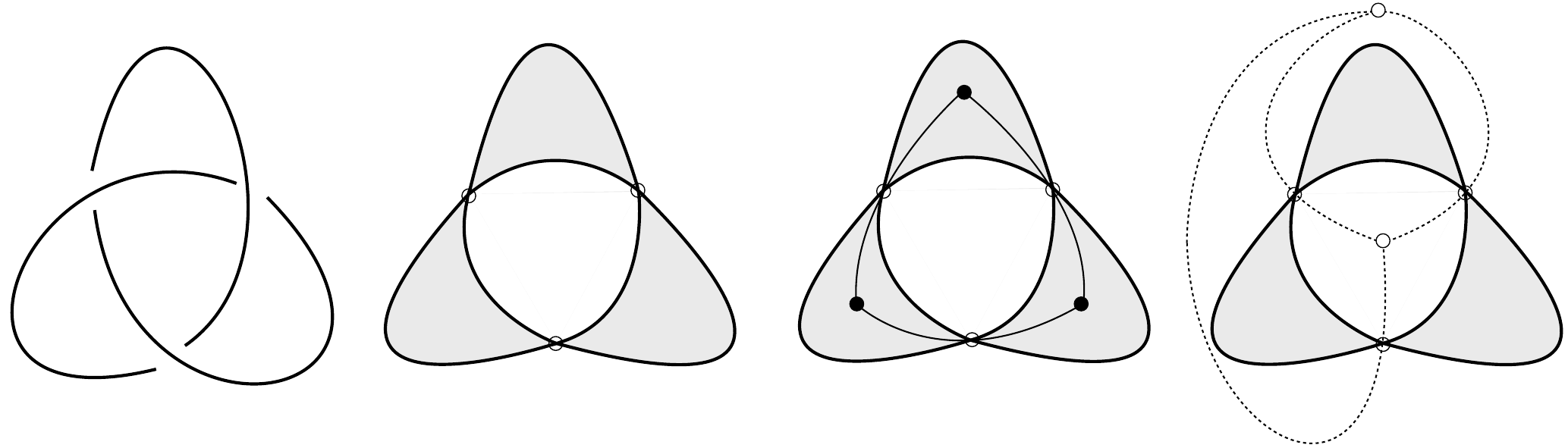}
\caption{(From left to right) A diagram of the Trefoil; its shadow with 2-colored faces (vertices on white crossed circles); the corresponding Black graph (bold edges and black circles); and White graph (dotted edges and white circles).}
\label{fig1}
\end{figure}

Given such a coloring, we can define two graphs, one on the faces of each color. Let $B_D$ denote the graph with black faces as its vertices and where two vertices are joined if the corresponding faces share a vertex ($B_D$ is called the {\em checkerboard graph} of $D$). We define the graph $W_D$ on the white faces of the shadow analogously.  We notice that $B_D$ (and $W_D$) may have multiple edges. The two graphs $B_D$ and $W_D$ are called {\em faces graphs} of the shadow. Since the shadow of a knot is connected, $B_D$ and $W_D$ are also connected, and it is not hard to see that $W_D=B_D^*$ and $B_D=W_D^*$, that is, the two faces graphs are duals of each other. 
\smallskip

Recall that the {\em medial graph} of a planar graph $H$, denoted by $med(H)$ is the graph obtained by placing one vertex on each edge of $H$ and joining two vertices if the corresponding edges are consecutive on a face of $H$. We notice that $med(H)$ is 4-regular since each edge is shared by exactly two faces. Multiple edges can arise in $med(H)$ when $H$ has faces of length two. We notice that $med(B_D)$ and $med(W_D)$ are the same (since $B_D$ and $W_D$ are duals) and that $med(W_D)$ is exactly the shadow of $D$.
\smallskip

An \textit{edge-signed} planar graph, denoted by $(G,\chi_E)$, is a planar graph $G(V,E)$ equipped with a signature on its edges $\chi_E:E\rightarrow \{+,-\}$. We will denote by $-\chi_E$ the signature of $G$ satisfying $-\chi_E(e)=-(\chi_E(e))$ for every $e\in E(G)$. 
\smallskip

Given a crossing of the link diagram, we designate this crossing as {\em positive} or {\em negative} according to the {\em left-over-right} and {\em right-over-left} rules from the point of view of black around the crossing; see Figure \ref{fig2} (a). The latter induces an opposite signing on each crossing by the same rules but now from the point of view of white around the crossing; see Figure \ref{fig2} (b).

\begin{figure}[H]
\centering
\includegraphics[width=.58\linewidth]{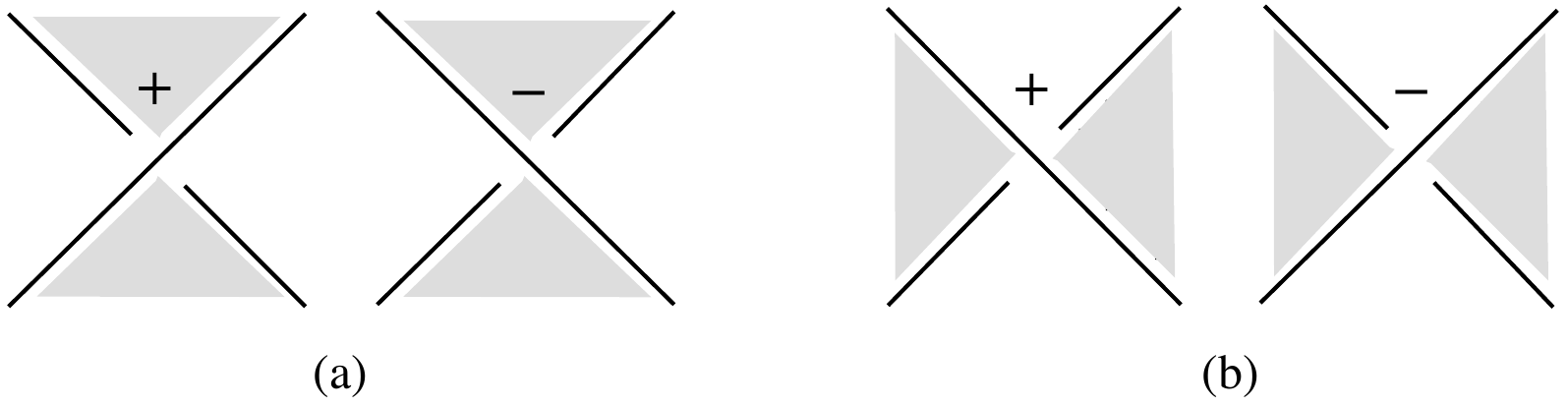}
\caption{(a) Left-over-right rule from black point of view. (b) Right-over-left rule from white point of view.}
\label{fig2}
\end{figure}


If the crossing is positive, relative to the black faces, then the corresponding edge is declared to be positive in $B_D$ and negative in $W_D$. Therefore, in this fashion, a link diagram $D$ determines a dual pair of signed planar graphs $(B_D,\chi_E)$ and $(W_D,\chi_E)$, where the signs on edges are swapped when moving to the dual.  

\begin{remark}\label{ram;blackwhite} A link diagram can be uniquely recovered from either $(B_D,\chi_E)$ or $(W_D,\chi_E)$.
\end{remark}

We thus have that, given an edge-signed planar graph $(G,\chi_E)$, we can associate to it (in a canonical way) a link diagram $D(L)$ such that $(B_D,\chi_E)$ (and $(W_D,\chi_E)$) determines $(G,\chi_E)$. The graph $G$ is called the {\em Tait graph} of the link $L$ with diagram $D$. The construction is easy. We just consider the medial graph of $G$ with signatures on its vertices (induced by the edge-signature $\chi_E$ of $G$). The desired diagram, denoted by $D(G,\chi_E)$, is obtained by determining the under/over pass at each crossing according to Left-over-right (or Right-over-left) rule associated to the sign of the corresponding edge of $G$; see Figure \ref{fig3}.

\begin{figure}[H]
\centering
\includegraphics[width=.6\linewidth]{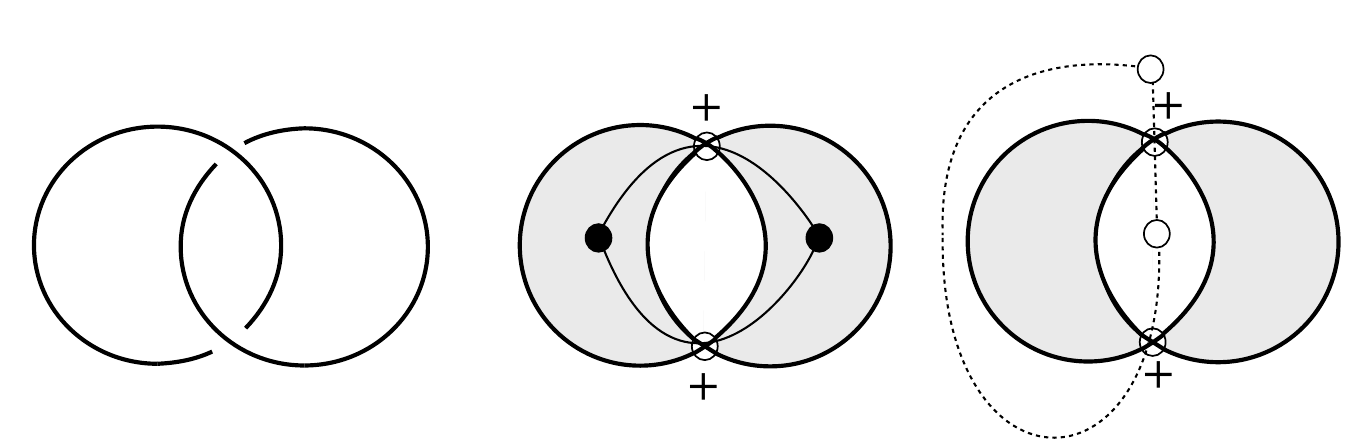}
\caption{(From left to right) diagram $D$ of link $2_1$; signed Black graph $(B_D,\chi_E)$; and signed White graph $(W_D,\chi_E)$.}
\label{fig3}
\end{figure}

A link $L$ is {\em alternating} if it admits a diagram $D$ such that the crossings alternate under/over while we go through the link. We notice that a link $L$ is alternating if and only if $L$ admits a diagram $D(G,\chi_E)$ with $\chi_E(e)=+$ (or $\chi_E(e)=-$) for all edge $e\in E(G)$.
\smallskip

Two links $L_1$ and $L_2$ are {\em equivalent} if there is an orientation-preserving homeomorphism 
$\varphi :\mathbb{R}^3\rightarrow  \mathbb{R}^3$ with $\varphi(L_1)=L_2$. We shall denote by $[L]$ (called {\em link-type}) the class of links equivalent to $L$.
\smallskip

\section{Novel approach}\label{sec:FH}

Let $G=(V,E)$ be a connected planar graph with $n=|E|$ and let $\mathcal{T}_G$ be the set of spanning trees of $G$.
Recall that $f_G$ is the Boolean function having $\support(f_G)=\{ {\bf v}_T \in \fefe_2^n \ \vert \  T\in \mathcal{T}_G \}$ 
where ${\bf v}_T=(v_1,\dots ,v_n)$ is the characteristic vector of the tree $T$, that is, $v_i=1$ if edge $i\in T$ and 0 otherwise. Finally recall that the {\em $FH_G$-polynomial} of $G$ is defined as $FH_G(x_1,\ldots , x_n)=\widehat f_G(x_1,\ldots , x_n)$.  

\begin{lemma}\label{res}
Let $G$ be a graph with $n$ edges and let $\w{u}\in\f{2}{n}$. Then,
$$
FH_G(\w{u}) = \big|\{T\in\mathcal{T}_G\mid |T\cap A_{\w{u}}| \equiv 0 \ (\bmod \ 2)\}\big| - \big|\{T\in\mathcal{T}_G\mid |T\cap A_{\w{u}}| \equiv 1 \ (\bmod \ 2)\}\big|
$$
where $A_{\w{u}} = \{i\in E\mid x_i = 1\}$.
\end{lemma}

\begin{proof} Let $\w{u}\in\f{2}{n}$. By definition, the Fourier--Hadamard transform associated to a Boolean function with support $\support_{\mathcal{T}_G}$ is given by
$$
FH(\w{u}) = \sum_{\w{t}\in \support_{\mathcal{T}_G}} (-1)^{\w{u}\cdot\w{t}} = \sum_{\substack{\w{t}\in \support_{\mathcal{T}_G}\\\w{t}\cdot\w{u} = 0}} 1 + \sum_{\substack{\w{t}\in S_{\mathcal{T}_G}\\\w{t}\cdot\w{u} = 1}} (-1).
$$
We thus have that $\w{u}\cdot\w{t} = 0$ if and only if $A_{\w{u}}\cap T$ is even where $T$ is the spanning tree such that $\w{t}$ is the characteristic vector of $T$. The result follows.
\end{proof}
\smallskip

{\em Proof of Theorem \ref{theo:main}.} Let  $(G,\chi_E)$ be an edge-signed connected planar graph and let $L$ be the link arising from $(G,\chi_E)$.
\smallskip

We claim that $\{T\in\mathcal{T}_G\mid |T\cap A_{\w{v}}| \equiv 0 \ (\bmod \ 2)\}$ with $\w{v}=(v_1,\ldots ,v_n)$ where $v_i=\frac{1-\chi_E(i)}{2}$ is exactly the set of all the positive-spanning trees of $G$. Indeed, by definition, $v_i=1$ if $\chi_E(i)=-$ and 0 otherwise. Thus, the set $T\cap A_{\w{v}}$ contains all negative-signed edges contained in the tree $T$. Therefore, $T$ is a positive-signed tree if and only if $|T\cap A_{\w{v}}| \equiv 0 \ (\bmod \ 2)$. Analogously, the set $\{T\in\mathcal{T}_G\mid |T\cap A_{\w{v}}| \equiv 1 \ (\bmod \ 2)\}$
is exactly the set of all the negative-spanning trees of $G$.

By combining, Lemmas \ref{lem:main} and \ref{res} we have

$$\begin{array}{ll}
\dett(L)&= \big|\# \{\text{positive-spanning trees of } G\}-\# \{\text{negative-spanning trees of } G\}\big|\\
& = \big| \#\{T\in\mathcal{T}_G\mid |T\cap A_{\w{v}}| \equiv 0 \ (\bmod \ 2)\} - \#\{T\in\mathcal{T}_G\mid |T\cap A_{\w{v}}| \equiv 1 \ (\bmod \ 2)\}\big| \\
& = \big| FH_G(\w{v}) \big|.
\end{array}$$
\hfill$\square$

\subsection{Computing $FH_G$-polinomial}  We now focus our attention on the calculation of $FH_G(\w{x})$. We will first give an explicit description of the $FH_G$-polynomial.






\begin{proposition}\label{baseform} Let $G=(V,E)$ be a graph with $n=|E|$ and let $\w{x}\in\f{2}{n}$. Then,
$$
FH_G(x_1,\ldots,x_n) = \sum_{T\in\mathcal{T}_G} \left(\prod_{i\in E(T)} (1-2x_i)\right).
$$
\end{proposition}

\begin{proof} The result follows from the Fourier-Hadamard transform formula by using that $(-1)^{x_i} = 1-2x_i$ for any $x_i\in\f{2}{}$. Indeed,
$$
FH_G(\w{x}) =\widehat{f}_{G}(\w{x}) = \sum_{\w{t}\in S_{\mathcal{T}_G}} (-1)^{\w{x}\cdot\w{t}} = \sum_{T\in \mathcal{T}_G} \left(\prod_{i\in E(T)} (-1)^{x_i} \right) = \sum_{T\in \mathcal{T}_G} \left(\prod_{i\in E(T)} (1-2x_i) \right).
$$
\end{proof}

We now propose a recursive method to calculate $FH_G(\w{x})$. Recall that the graph obtained by deleting the edge $e$ from $G$ is denoted by $G\setminus e$ and that $e=(uv)$ is {\em contracted} if its endpoints are identified as a single vertex and the edge $e$ is deleted. The resulting graph is denoted by $G / e$.  An {\em isthmus} is an edge whose deletion increases the graph's number of connected components. Finally, a {\em loop} is an edge that connects a vertex to itself.

\begin{lemma} Let $G$ be a graph consisting of one edge, say $e$. Then,
$$FH_G(x_e)=\left\{\begin{array}{ll}
1-2x_e & \text{ if } e \text{ is an isthmus},\\
1 & \text{ if } e \text{ is a loop}.\\
\end{array} \right.$$
\end{lemma}

\begin{proof}
If $e$ is an isthmus, then there is only one spanning tree consisting of edge $e$, and the result follows by Proposition \ref{baseform}. If $e$ is a loop then there si not spanning tree and thus $f_{G} = \w{0}$ and the result follows by definition of $FH_G$ polynomial.
 \end{proof}

 \begin{proposition}\label{recformula1} Let $G$ be a graph and let $e$ be either a loop or an isthmus of $G$. Then,
 $$
 FH_G(\w{x}) = FH_G(x_e)\ FH_{G\setminus e}(\w{x'})
 $$
 where $\w{x'}$ is obtained from $\w{x}$ from which the entry $x_e$ corresponding to edge $e$ is deleted.
\end{proposition}

\begin{proof}
If $e$ is a loop, then it is not in any spanning tree, and so the set of spanning trees is the same for both $G$ and $G\setminus e$, implying thus that the polynomial is the same. If $e$ is an isthmus, then $e$ is in every base, and the result follows by Proposition \ref{baseform}.
 \end{proof}

 \begin{proposition}\label{recformula}  Let $G$ be a graph and let $e$ be an edge neither a loop nor an isthmus.
 Then,
 $$
 FH_G(\w{x}) = FH_{G\setminus e}(\w{x'}) + (1-2x_e)FH_{G/e}(\w{x'})
 $$ 
  where $\w{x'}$ is obtained from $\w{x}$ from which the entry $x_e$ corresponding to edge $e$ is deleted.
 \end{proposition}

 \begin{proof} Let $e$ be an edge in $G$.
 Let $\mathcal{T}_G^{\text{ w/o } e}$ (resp. $\mathcal{T}_G^e$) be the set of spanning trees in $G$ not containing $e$ (resp. containing $e$). Since $e$ is neither a loop nor an isthmus then the set of spanning trees in $G\setminus e$ is given by $\mathcal{T}_G^{\text{ w/o}\  e}$ while an spanning tree $T$ of $G / e$ is  of the form $T\setminus e$ with $e\in T\in\mathcal{T}_G$. By Proposition \ref{baseform}, we have
 $$\begin{array}{ll}
 FH_G(\w{x}) & = \sum\limits_{T\in\mathcal{T}_G}\prod\limits_{i\in T}(1-2x_i) \\
 & = \sum\limits_{T\in\mathcal{T}_G^{\text{ w/o } e}}\prod\limits_{i\in T}(1-2x_i) + \sum\limits_{T\in\mathcal{T}_G^{e}}\prod\limits_{i\in T}(1-2x_i)\\
 &= FH_{G\setminus e}(\w{x}) + (1-2x_e)\sum\limits_{T\in\mathcal{T}_G^e}\prod\limits_{\substack{i\in T \\ i \neq e}}(1-2x_i)\\
 & = FH_{G\setminus e}(\w{x}) + (1-2x_e) FH_{G/e}(\w{x}).
 \end{array}$$
 \end{proof}

\begin{example} Let $G$ be the graph given on the top of Figure \ref{figFH}. 

\begin{figure}[H]
\centering
\includegraphics[width=1\linewidth]{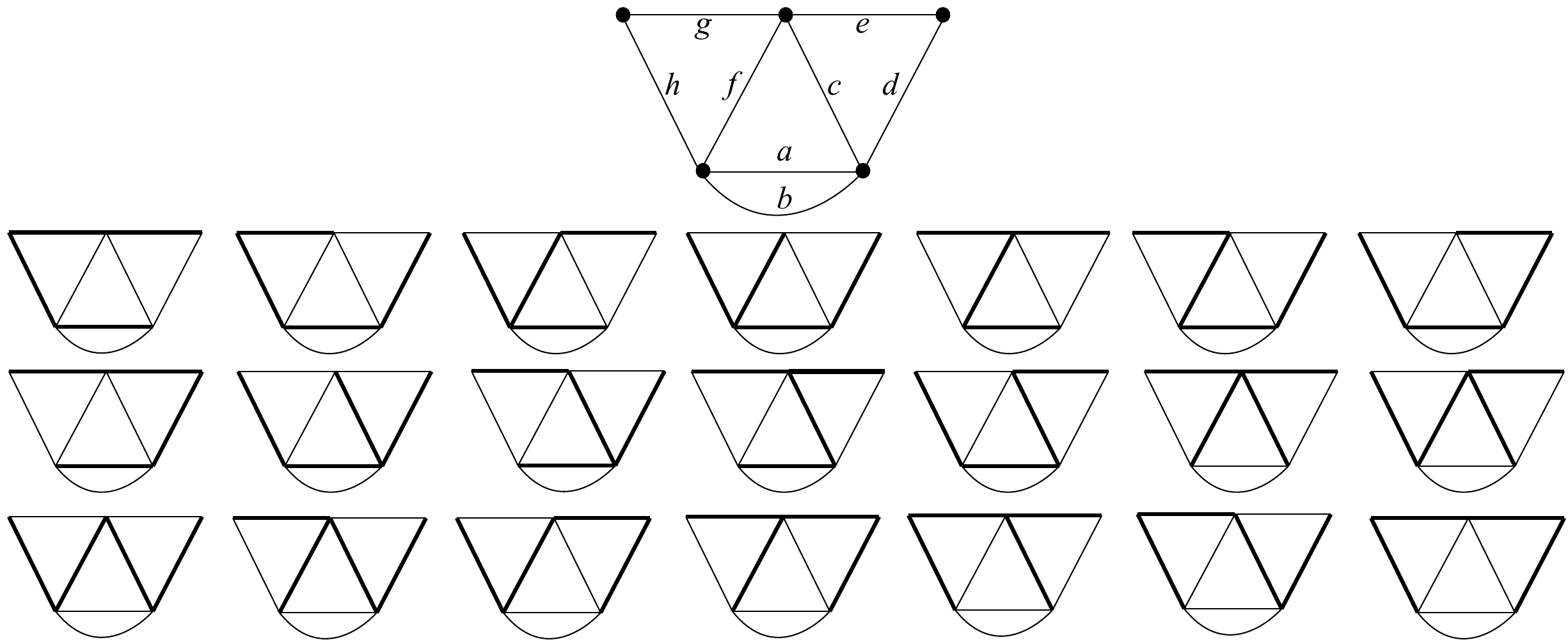}
\caption{Graph $G$ and 21 spanning trees (the other 12 can be obtained from the first 12 trees by replacing edge $a$ by edge $b$).}\label{figFH}
\end{figure}

By Proposition \ref{baseform}, we have,

{\tiny
 $$\begin{array}{ll}
 FH_G(a, b, c, d, e, f, g, h)&=(1-2a)(1-2h)(1-2g)(1-2e)+(1-2a)(1-2h)(1-2g)(1-2d)+(1-2a)(1-2h)(1-2f)(1-2d)\\
&+(1-2a)(1-2h)(1-2f)(1-2e)+(1-2a)(1-2f)(1-2g)(1-2d)+(1-2a)(1-2f)(1-2g)(1-2e)\\
&+(1-2a)(1-2d)(1-2e)(1-2h)+(1-2a)(1-2d)(1-2e)(1-2g)+(1-2a)(1-2d)(1-2c)(1-2h)\\
&+(1-2a)(1-2d)(1-2c)(1-2g)+(1-2a)(1-2c)(1-2e)(1-2g)+(1-2a)(1-2c)(1-2e)(1-2h)\\
&+(1-2b)(1-2h)(1-2g)(1-2e)+(1-2b)(1-2h)(1-2g)(1-2d)+(1-2b)(1-2h)(1-2f)(1-2d)\\
&+(1-2b)(1-2h)(1-2f)(1-2e)+(1-2b)(1-2f)(1-2g)(1-2d)+(1-2b)(1-2f)(1-2g)(1-2e)\\
&+(1-2b)(1-2d)(1-2e)(1-2h)+(1-2b)(1-2d)(1-2e)(1-2g)+(1-2b)(1-2d)(1-2c)(1-2h)\\
&+(1-2b)(1-2d)(1-2c)(1-2g)+(1-2b)(1-2c)(1-2e)(1-2g)+(1-2b)(1-2c)(1-2e)(1-2h)\\
&+(1-2c)(1-2f)(1-2e)(1-2g)+(1-2c)(1-2f)(1-2e)(1-2h)+(1-2c)(1-2f)(1-2d)(1-2g)\\
&+(1-2c)(1-2f)(1-2d)(1-2h)+(1-2f)(1-2h)(1-2e)(1-2d)+(1-2f)(1-2h)(1-2e)(1-2d)\\
&+(1-2c)(1-2e)(1-2g)(1-2h)+(1-2c)(1-2d)(1-2g)(1-2h)+(1-2h)(1-2g)(1-2e)(1-2d).
\end{array}$$
 }

If we take the edge-signature $\chi(a)=\chi(b)=\chi(c)=\chi(d)=\chi(e)=\chi(f)=\chi(g)=\chi(h)=+1$ then $FH_G(0,0,0,0,0,0,0,0)=33$ which corresponds to the number of spanning trees of $G$ and thus giving the determinant of the corresponding alternating link.  Let us now consider the edge-signature $\chi(a)=\chi(b)=-$ and $\chi(c)=\chi(d)=\chi(e)=\chi(f)=\chi(g)=\chi(h)=+$. The corresponding non-alternating link is given in figure \ref{figsign}. 

\begin{figure}[H]
\centering
\includegraphics[width=.3\linewidth]{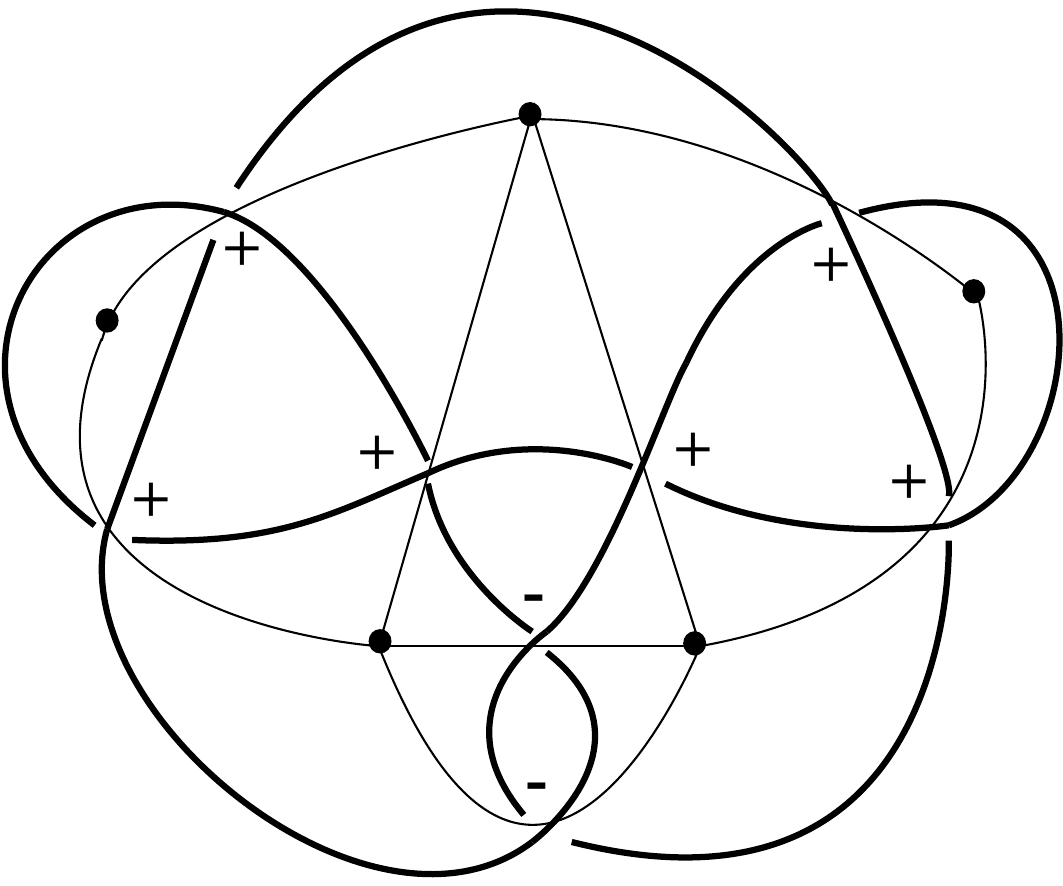}
\caption{Knot $8_{21}$ (in bold) and its Tait's graph (light edges).}\label{figsign}
\end{figure}

In this case  $\big|  FH_G(1,1,0,0,0,0,0,0) \big|  =15=\dett(8_{21})$. We notice that the latter corresponds to the difference between the 24 negative-spanning trees (those containing either edge $a$ or $b$) and the 9 positive-spanning trees (the last 9 trees in Figure \ref{figFH}). 

Finally, the set of the determinants, obtained by one of the $2^8$ edge-signatures, is given in Table \ref{table:det}.

\begin{table}[H]
\begin{tabular}{c|c}
\hline
$\dett(L)$ &  \# of edge-signatures \\
\hline
1 & 46\\
3 & 44\\
5 & 14\\
7 & 7\\
9 & 9\\
11 & 2\\
13 & 2\\
15 & 1\\
19 & 2\\
33 & 1\\
\end{tabular}

\caption{The determinants obtained from the 128 vectors beginning with 1 (the other 128 vectors, beginning with 0, yields to exactly the same set of determinants).}\label{table:det}
\end{table}
\end{example}

\section{Centrally symmetric links}\label{sec:symmetry}

We quickly recall some background on maps needed for the proof of Theorem \ref{th:det:sym}, see \cite{MRR} for further details.

\subsection{Maps background} 
A {\em map} of $G=(V,E,F)$ is the image of an embedding of $G$ into $\mathbb{S}^2$ where the set of vertices are a collection of distinct points in $\mathbb{S}^2$ and the set of edges are a collection of Jordan curves joining two points in $V$ satisfying $\beta\cap\beta'$ is either empty or a point in the endpoints for any pair of Jordan curves $\beta$ and $\beta'$.  Any embedding of the topological realization of $G$ into $\mathbb{S}^2$ partitions the 2-sphere into simply connected regions of $\mathbb{S}^2\setminus G$ called the {\em faces} $F$ of the embedding. A map $G$ is called {\em self-dual} if there is a bijection from $V$ and $F$ to $V$ and $F$ which reverses inclusion.

Let 
\[
\begin{array}{llcl}
\alpha: & \mathbb{S}^2\subset \rere^3& \rightarrow & \mathbb{S}^2\subset\rere^3\\ 
& x & \mapsto & -x
\end{array}
\]
be  the {\em antipodal} function.

Notice that $\alpha$ is a homeomorphism of $\mathbb{S}^2$ into itself without fixed points. We say that $Y\subseteq\mathbb{S}^2$ is {\em antipodally symmetric} if $\alpha(Y)=Y$. We say that $G$ is an {\em antipodally symmetric} map if it admits an embedding in $\stw$ such that $\alpha(G)=G$.
If $G$ is an antipodally symmetric map and $v\in V(G)$ then its {\em antipodal} vertex is given by $\alpha(v)=-v$. We call $(v,-v)$ {\em antipodal pair} of vertices. 

\begin{remark}\label{rem:2-aut} If $G$ is an antipodally symmetric map then its number of faces must be even. Moreover, the function $\alpha$ naturally matches the pairs of {\em antipodal faces}, say $f$ and $\alpha(f)$ (we may refer to $\alpha(f)$ as {\em $f$-antipodal}). This pairing is a permutation of the faces that turns out to be an automorphism of $G^*$ (that is, $\alpha\in Aut(G)$, and thus $G^*$ is also antipodally symmetric). 
\end{remark}

\begin{proposition} \cite[Proposition 1]{MRR} Let $G$ be an antipodally symmetric map where its faces are 2-colored properly (that is, two faces sharing an edge have different colors). Then, if one pair of antipodal faces has the same (resp. different color) then all pairs of antipodal faces have the same (resp. different color). 
\end{proposition}


 
We denote by $(G,C_F,\chi_V)$ the a {\em face-colored} $C_F:F\rightarrow \{black, white\}$ and {\em vertex-signed} $\chi_V:V\rightarrow \{+,-\}$ map $G$. We denote by $\chi_V^+$ the vertex signatures where all the signs are the same.
We say that an automorphism $\sigma(G)\in Aut(G)$ is \textit{color-preserving} (resp. \textit{color-reversing}) if each pair of faces $f$ and $\sigma(f)$ have the same (resp. different) color. Similarly, $\sigma$ is said to be \textit{sign-preserving} (resp. \textit{sign-reversing}) if each pair of vertices $v$ and $\sigma(v)$ have the same (resp. different) sign.

\begin{remark}\label{re;anti} (a) In the case when $(G,C_F,\chi_V)$ is an antipodally symmetric map, the automorphism $\alpha$ can be either color-preserving or color-reversing, see Figure \ref{fig14}.

\begin{figure}[H]
\centering
\includegraphics[width=.5\linewidth]{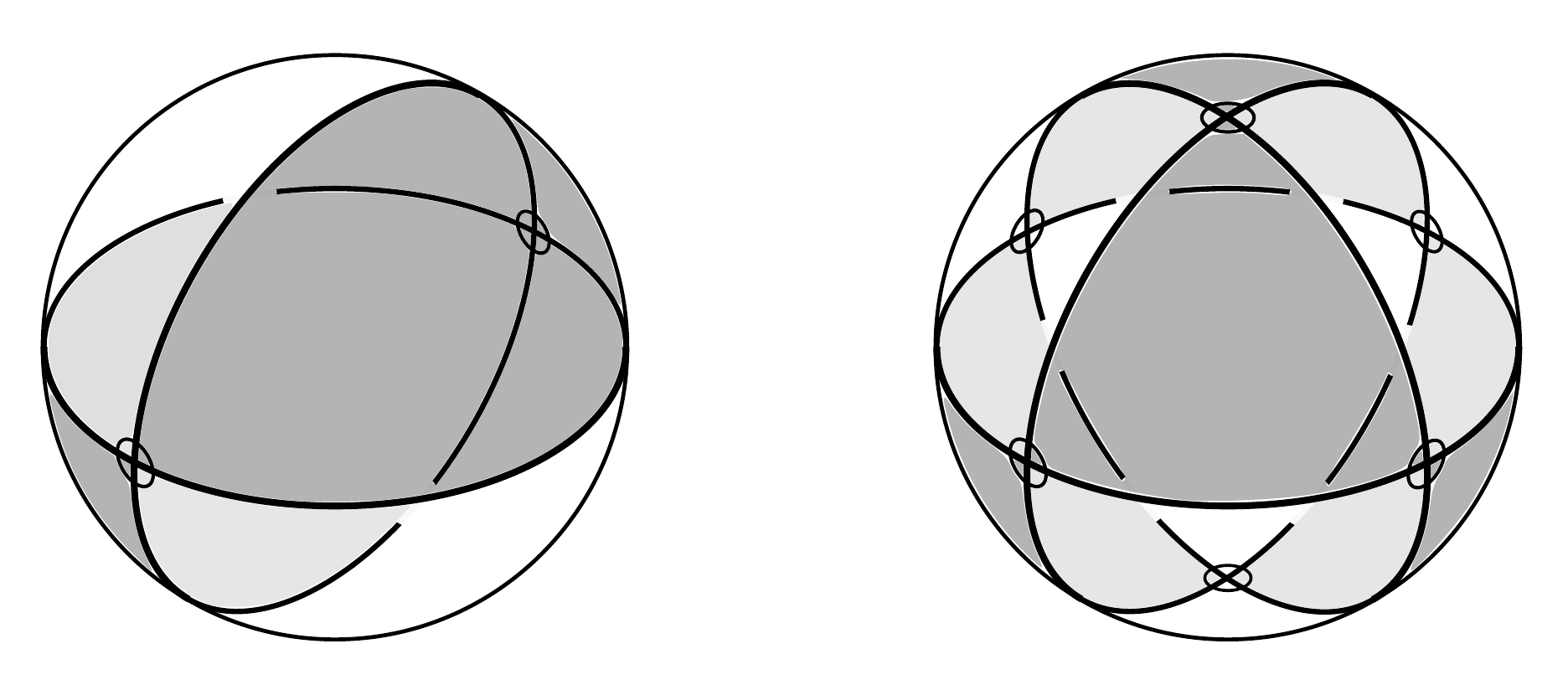}
\caption{(Left) An antipodally symmetric map where $\alpha$ is color-preserving. (Right) An antipodally symmetric map where the $\alpha$ is color-reversing.}\label{fig14}
\end{figure}

(b) If $\alpha$ is color-preserving then the number of black faces (and white faces) is even. Indeed, suppose that there is a black face $F$ such that $\alpha(F)=F$. Since any face of the diagram has a connected boundary then it is homeomorphic to a disc ($F$ could be a band with borders of $\stw$ which is also homeomorphic to a disc). Therefore, $\alpha$ sends a disc to itself and, by Brouwer's theorem, $\alpha$ would have a fixed point which is a contradiction since the antipodal function $\alpha$ is fixed-point free. 
\end{remark}

\subsection{Special embedding construction}
Let $(G,\chi_E)$ be an edge-signed map. We have that the map $(med(G), C_F, \chi_V)$ (where $\chi_V$ is induced by $\chi_E$) determines, in a canonical way, a diagram $D(G,\chi_E)$ of a link $L$.
\smallskip

We shall construct a specific embedding of $L$ in $\rth$, denoted by $\embed(G,\chi_E)$, by modifying (locally) $D(G,\chi_E)$ around each crossing as follows. For each vertex of $med(G)$, let $S_v$ be the sphere centered at $v$ and radius small enough to intersect only the edges incident to $v$ (these four points close to $v$). Then, replace  the piece of arc of the diagram passing over (resp. passing under) by the half meridians of $S_v$ according to the {\em crossing sphere rules}, see Figure \ref{crossing-spheres}. 

 \begin{figure}[H]
    \centering
    \includegraphics[width=.8\textwidth]{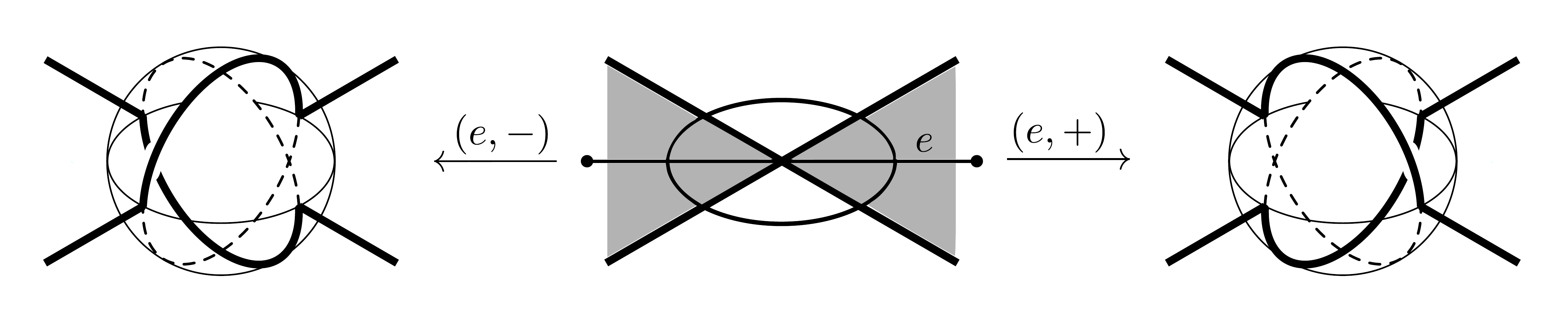}
    \caption{The crossing spheres rules.}
    \label{crossing-spheres}
\end{figure} 

The rest of the diagram $D(G,\chi_E)$ remains the same in $\mathbb{S}^2$; see Figure \ref{fig16}.

\begin{figure}[H]
\centering
\includegraphics[width=.5\linewidth]{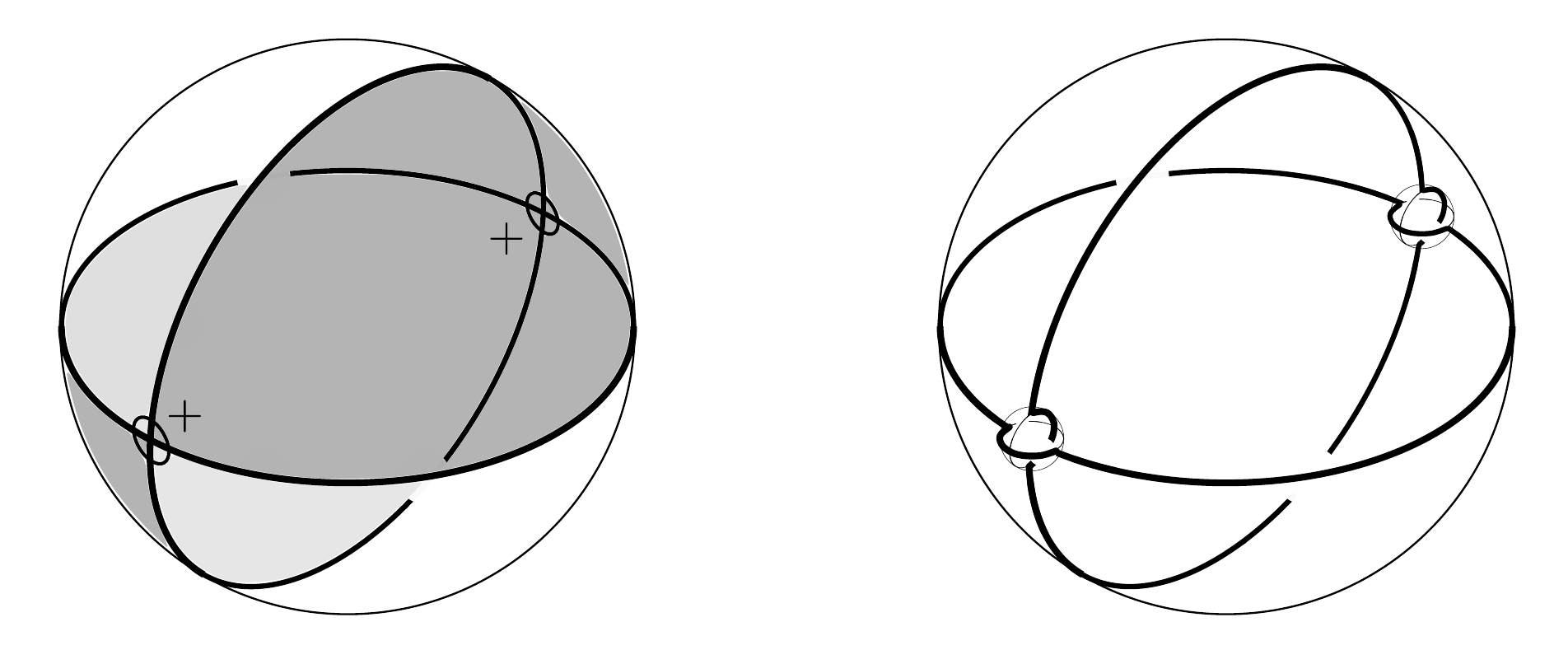}
\caption{(Left) An antipodally symmetric map of the graph $H$ on 2 vertices and four parallel edges. This map corresponds to a shadow of the Hopf link $2_1$, see Figure \ref{fig4}. In this case, $\alpha$ is color-preserving and sign-preserving. (Right) The embedding $\embed(2_1)$.}
\label{fig16}
\end{figure}

\subsection{Characterization of centrally symmetric links}

We may present a combinatorial characterization of centrally symmetric links.

\begin{theorem} \label{thm:centralsymm} 
		 A link $L$ is centrally symmetric if and only if
		 there is an edge-signed map $(G,\chi_E)$ in $\mathbb S^2$ satisfying the following conditions:
		 \begin{enumerate}
		 	\item The map $(med(G), C_F, \chi_V)$ determines a diagram $D(G,\chi_E)$ representing $L$.
		 	\item $med(G)$ is antipodally symmetric in $\mathbb S^2$ (say, realized by $\alpha$).
		 	\item $\alpha$ is either color-preserving and sign-reversing  or color-reversing and sign-preserving.
		 \end{enumerate}
\end{theorem}

\begin{proof}  ({\em Sufficiency}) Let $(med(G),C_F,\chi_V)$ be an antipodally symmetric medial map realized by $\alpha$. We consider the embedding $\embed(G,\chi_E)$. It can be checked that if $\alpha$ is either color-preserving and sign-reversing  or color-reversing and sign-preserving then the piece of arc of the diagram passing over (resp. passing under) around vertex $v$ correspond to the piece of arc of the diagram passing over (resp. passing under) around the antipodal  vertex $\alpha(v)$, see Figure \ref{fig:2antisym}.

\begin{figure}[H]
\centering
\includegraphics[width=0.7\linewidth]{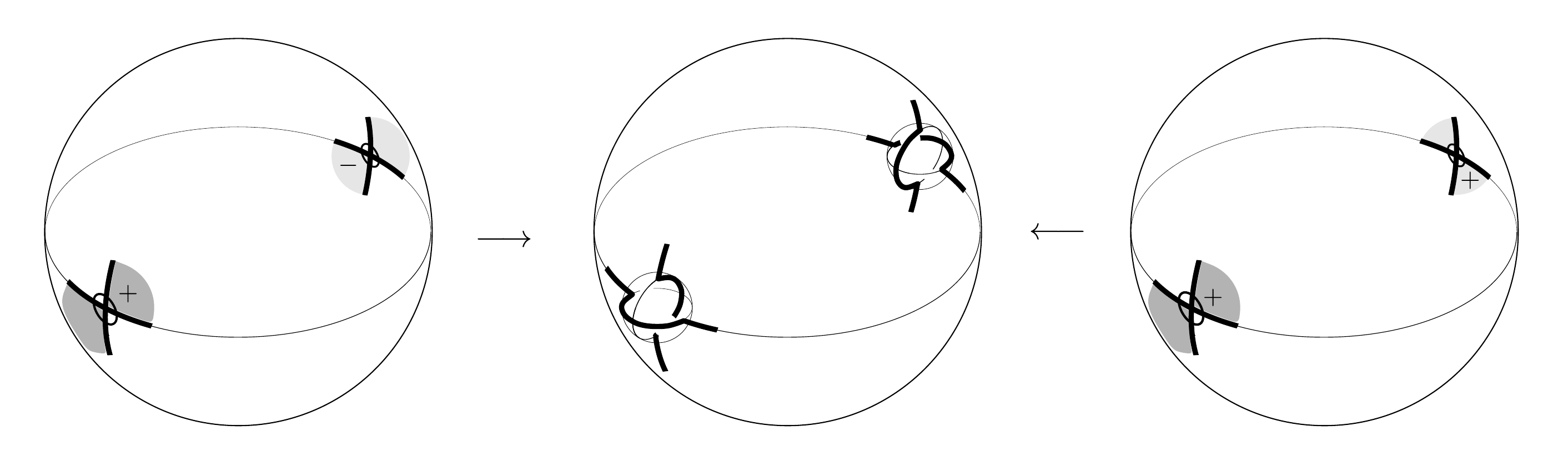}
\caption{(Left) Antipodal pair of vertices of $med(G)$ with $\alpha$ color-preserving and sign-reversing. (Right)  Antipodal pair of vertices of $med(G)$ with $\alpha$ color-reversing and sign-preserving. (Center) The local modifications around the antipodal pair of vertices following the color-sphere rules.}
\label{fig:2antisym}
\end{figure}

We thus have that  $\embed(G,\chi_E)$ is centrally symmetric. 
\smallskip

 ({\em Necessity}) Suppose that $L$ is centrally symmetric. Hence, $L$ admits an embedding, say  $\hat L$, in ${\mathbb R^3}$ with $c(\hat L)=\hat L$. We claim that there is $(med(G),C_F,\chi_V)$ for some map $(G,\chi_E)$ inducing a link isotopic to $\hat L$ and with the desired coloring and sign conditions. The latter can be done by using the same construction as the one used in the proof of \cite[Theorem 1]{MRR1}. The procedure goes as follows. First, $\hat L$ can be thought of as a special embedding  $\embed(G,\chi_E)$ for some edge-signed map $(G, \chi_E)$. For this, one may take the radial projection $p(\hat L)$ from $\hat L$ to $\stw$, that is, if we let $r(x)$ be the ray emitting from the origin passing through $x$ then

$$\begin{array}{lllc}
p: & \hat  L& \longrightarrow & \stw\\ 
& x & \mapsto & r(x)\cap\stw
\end{array}$$

Since $\hat L$ is centrally symmetric then $p(\hat  L)$ is clearly antipodally symmetric in $\stw$. We suppose that the projection $p(\hat  L)$ is regular in the sense that it avoids cusps and tangency points (this can be obtained by making some suitable local modifications to $p(\hat  L)$ done in a symmetric fashion in order to keep the symmetric antipodality of $p(\hat  L)$). A multiple intersection $y$ can be fixed by moving properly the piece of $p(\hat L)$ around $y$. We  end up with an antipodally symmetric projection without multiple points. This projection can be thus thought of as a 4-regular antipodally symmetric map (realized by $\alpha$) which, in turn, can be regarded as a medial map $med(G)$ for some map $G$. 
\smallskip

We notice that for each projected intersection $q$, we have the information what piece of $\hat L$ pass over/under the other the piece and symmetrically for $\alpha(q)$, keeping the central symmetry. Finally,  if we color the faces of $med(G)$ and sign its vertices (intersections) according with the latter information and respecting the crossing sphere rules then we have that the induced antipodally symmetric color-face vertex sign $(med(G),C_F,\chi_V)$ inducing a link isotopic to $L$. Moreover, the only way that the over/under crossing for each antipodal pair of vertices can satisfy the central symmetry is when $\alpha$ is either color-preserving and sign-reversing  or color-reversing and sign-preserving.
\end{proof}

\subsection{Proof of Theorem \ref{th:det:sym}.} We first need the following

\begin{lemma}\label{lem:parity} Let $(G,\chi_E)$ be an edge-signed map and suppose that $(med(G), C_F, \chi_V)$ is an antipodally symmetric map
(realized by $\alpha$). Let $D(L)$ be the link diagram determined by $(med(G), C_F, \chi_V)$. Then, the number of components of $D(L)$ is even if and only if $\alpha$ is color-preserving.
\end{lemma}

\begin{proof} We first claim that the components of $L$ can be oriented such that $\alpha$ preserves the orientation. Let $C_1$ and $C_2$ two components of $L$. If $\alpha(C_1)=C_2$ then any choice for the orientation of $C_1$ fixes the orientation of $C_2$ according to $\alpha$. Suppose that a component $C$ is {\em self-antipodal}, that is, $\alpha(C)=C$. We first show that any orientation of $C$ is preserved by $\alpha$.
\smallskip

\begin{claim} Let $C : \mathbb{R} / \mathbb{Z} \to \mathbb{S}^2$ be an antipodal component. Then, up to parametrization, $C$ has constant speed implying thus $\alpha(C(t))=C(t+ \frac{1}{2})$ for any $t \in \mathbb{R} / \mathbb{Z}$. 
\end{claim}

\begin{proof} Let $x = C(t)$ and $y = \alpha(x)$ be two antipodal points on $C$. Let $\gamma_1$ (resp. $\gamma_2$) the path from $x$ to $y$ (resp. from $y$ to $x$) following $C$. We claim that these paths are antipodal to each other. Indeed, suppose that there is a self-antipodal arc $\gamma$ (contained, say in $\gamma_1$) parametrized by $t\in [0,1]$. Then, around the neighborhood of $\gamma(0)$ and $\gamma(1)$, we must have that $\gamma(t) = \alpha(\gamma(1-t))$ for small values of $t$. Therefore, by connectivity, we obtain $\gamma(t) = \alpha(\gamma(1-t))$ for all $t \in [0,1]$. The latter implies that the point $\gamma(\frac{1}{2})$ is self-antipodal, which is a contradiction.


We have thus that $\gamma_1$ and $\gamma_2$ are antipodal and have the same length. Since $C=\gamma_1\cup\gamma_2$ and it has constant speed, we have $y = \alpha(x) = \alpha(C(t)) = C(t+\frac{1}{2})$. We conclude that any orientation of $C$ is thus preserved by $\alpha$.     
\end{proof}

Let us then assume that $D$ has an orientation preserved by $\alpha$. Now, each crossing has one of the two homotopy type illustrated in Figure \ref{fig:hom-type}. 

\begin{figure}[H]
\centering
\includegraphics[width=0.7\linewidth]{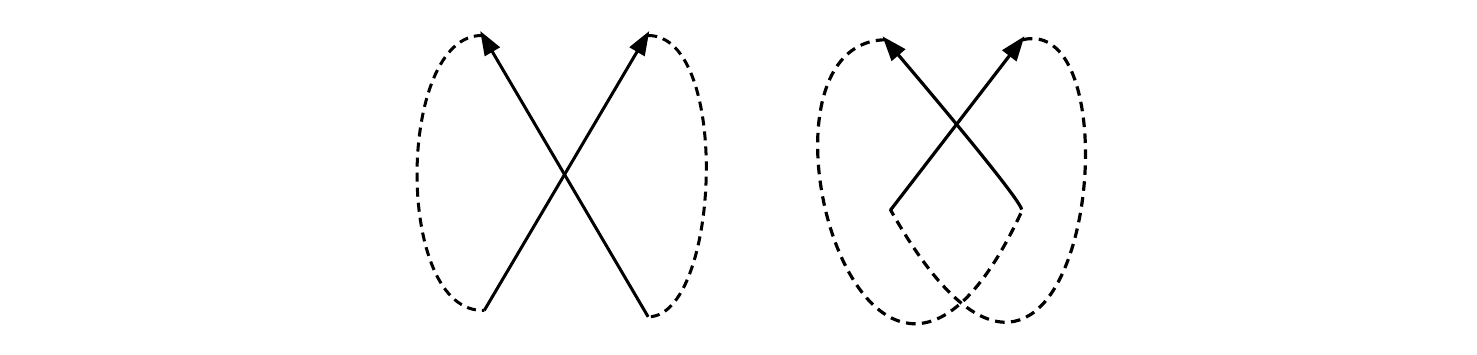}
\caption{Two homotopy types.}
\label{fig:hom-type}
\end{figure}

An {\em oriented-preserving} splitting of a crossing of $D(L)$ is the smoothing of the crossing preserving the orientation of $D(L)$, see Figure \ref{fig:smoth-or}. 

\begin{figure}[H]
\centering
\includegraphics[width=0.7\linewidth]{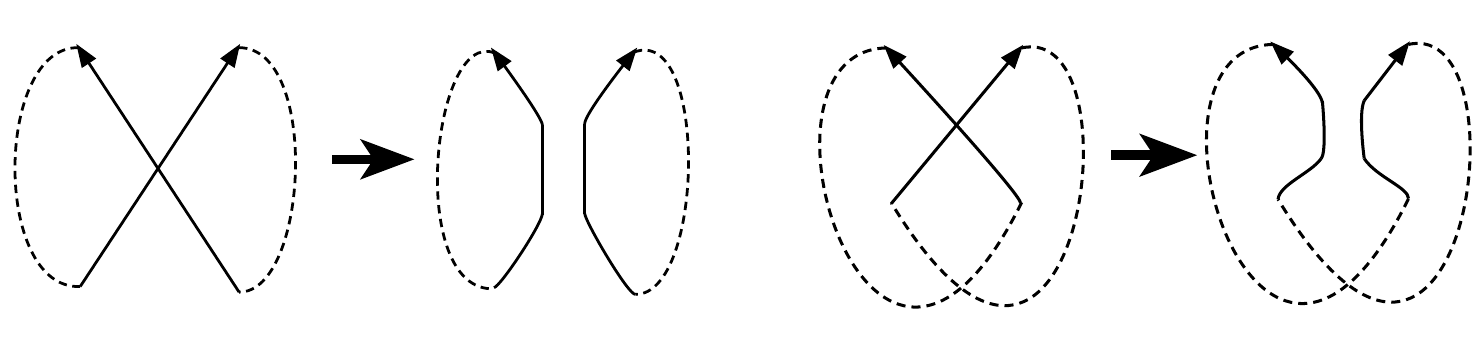}
\caption{(Left) Oriented-preserving splitting creating a new component. (Right) Oriented-preserving splitting reducing one component.}
\label{fig:smoth-or}
\end{figure}

We observe that each oriented-preserving splitting either creates one new component or reduces the number of components by one. Therefore, by applying an oriented-preserving splittings to two antipodal crossings, the parity of the number of components  of $D(L)$ is preserved. Moreover, the new link diagram is also centrally symmetric (by performing the two splittings in a symmetrical way) and the color-reversing (or the color-preserving) face coloring of $\alpha$ is also conserved. By carrying on oriented-preserving splittings, we end up with a set of antipodal symmetric disjoint simple closed curves in the sphere. We call this configuration the {\em base case}. 

\begin{claim}\label{claim:base}
The number of curves in the base case is even if and only if $\alpha$ is color-preserving. 
\end{claim}

\begin{proof}
({\em Necessity}) Showed in Remark \ref{re;anti} (b).
 \smallskip
 
({\em Sufficiency}) If $\alpha$ is color-reversing then there is an even number number of faces and therefore an odd number of curves.
\end{proof}

The result follows by induction on the number $m$ of antipodal pair of crossings where Claim \ref{claim:base} is the base case, that is, when $m=0$.
\end{proof}

In \cite[Question 2]{MRR}, it was asked the following

\begin{question}
Let $L$ be the alternating link arising from a map $(med(G), C_F, \chi_V^+)$ where $G$ is antipodally self-dual. Is it true that $L$ always has an odd number of components ?
\end{question}

Since $G$ is antipodally self-dual then $med(G)$ admits an antipodally symmetric embedding with $\alpha$  color-reversing, see \cite[Remark 3]{MRR}. The above question is thus answered affirmatively by Lemma \ref{lem:parity}.

\smallskip

We may now prove Theorem \ref{th:det:sym}.
\smallskip

{\em Proof of Theorem \ref{th:det:sym}.} Since $L$ is centrally symmentric, then, by Theorem \ref{thm:centralsymm}, there is a map $(G,\chi_E)$ such that $med(G)$ is antipodally symmetric (realized by $\alpha$) and $(med(G),C_F,\chi_V)$ determining a diagram representing $L$. Also, since the number of components of $L$ is even then, by Lemma \ref{lem:parity}, $\alpha$ is color-preserving. Therefore, again by Theorem \ref{thm:centralsymm},
$\alpha$ must be sign-reversing. 

W.l.o.g., we may suppose that the black faces of $med(G)$ correspond to the set of vertices $V(G)$ of $G$ and thus, by Remark \ref{re;anti} (b),  we have that $|V(G)|$ is even. As pointed out in Remark \ref{rem:2-aut}, $\alpha\in Aut(G)$. We claim that $\alpha$ gives a bijection between positive and negative spanning trees of $G$. Indeed, let $T$ be a spanning tree in $G$, since $\alpha$ is sign-reversing then $\chi(e)=-\chi(\alpha(e))$ (recall that $\chi(e)$ denotes the sign of edge $e$). Notice that $|E(T)|=|V(T)|-1=|V(G)|-1$ is odd since $|V(G)|$ is even. Hence,

$$\begin{array}{ll}
\sign(\alpha(T)) & = \prod\limits_{e\in E(T)} \chi(\alpha(e))\\
& = \prod\limits_{e\in E(T)} - \chi(e)\\
& = (-1)^{|E(T)|}\prod\limits_{e\in E(T)} \chi(e)\\
 &= - \prod\limits_{e\in E(T)} \chi(e)\\
& = -\sign(T).
\end{array}$$

Therefore, there are as many positive spanning trees as negative ones. The result follows by Lemma \ref{lem:main}.
\hfill$\square$

\section{Concluding remarks}

All the above results concerning the $FH_G$-polynomial can easily be extended in terms of {\em matroids} instead of graphs. 
In particular, the above recursive method to compute $FH_G(\w{x})$ (Proposition \ref{recformula}) behaves similarly as the so-called {\em Tutte polynomial} $t(M;x;y)$ associated to a matroid $M$. 
\smallskip

It is known that $t(M;1,1)$ counts the number of {\em bases} in $M$. If $M$ is a {\em graphic matroid}, that is, a matroid $M_G$ arising from a graph $G$, then $t(M_G;1,1)$ counts the number of spanning trees of $G$. It can easily be shown that  $FH_G(\w{0})=t(M_G;1,1)$.  


\appendix

\section{Determinant through the state model}\label{Append}

A nice combinatorial approach to calculate $\dett(L)$ for alternating diagrams was presented by Krebes \cite{Krebes} by using the state model for the Kauffman bracket. This model is helpful for our purpose, we thus quickly recall some main notions. We refer the reader to \cite{kauffman} for a detailed exposition on the Kauffman bracket.

Let $D$ be a connected link diagram. A {\em state} $S$ is a choice of {\em splitting} of type $L_0$ or $L_\infty$ of each crossing of $D$; see Figure \ref{fig4}.

\begin{figure}[H]
    \centering
    \includegraphics[width=0.4\textwidth]{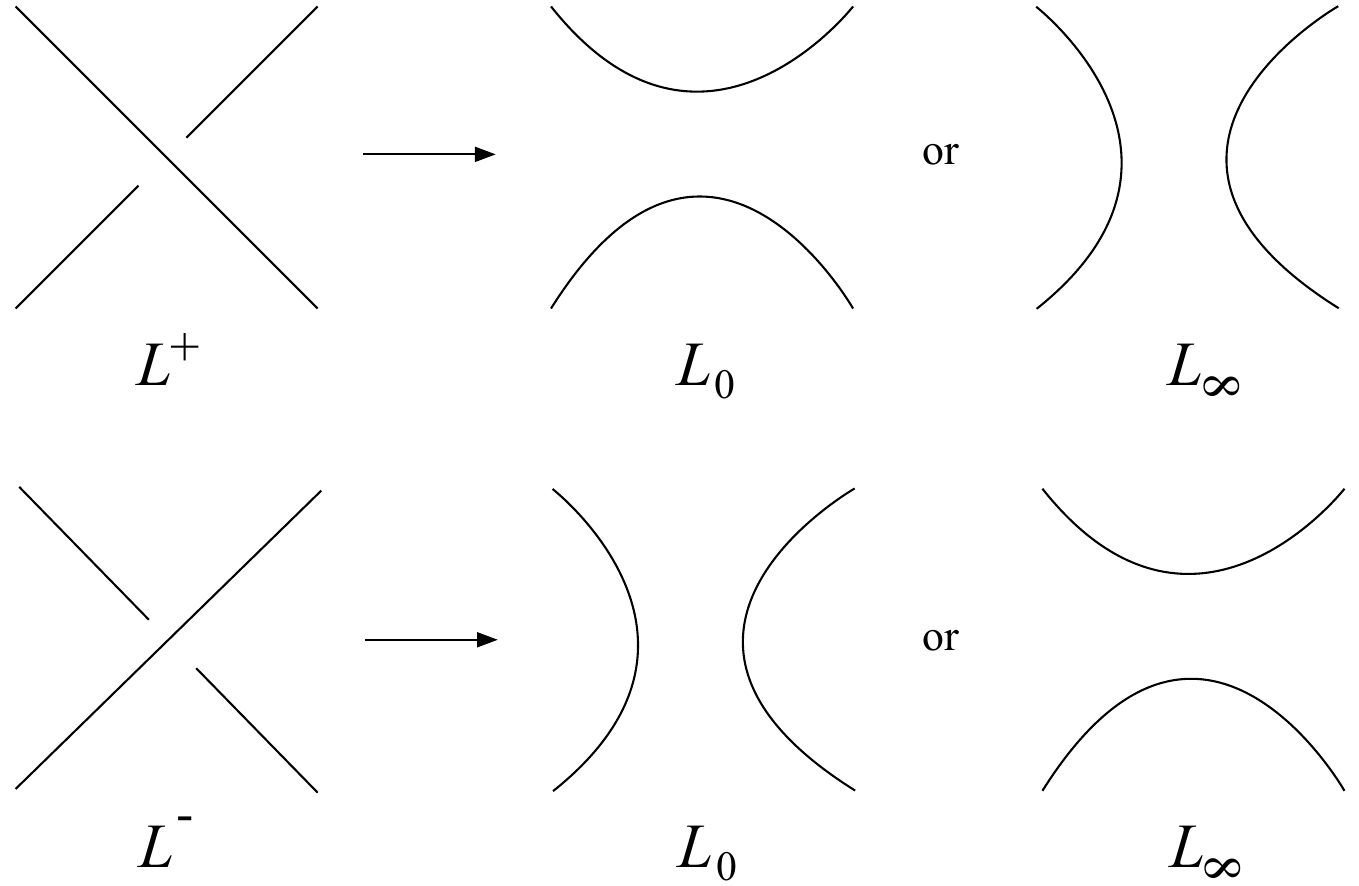}
    \caption{Splittings of type $L_0$ and $L_\infty$ depending on the sign associated to the crossing.}
    \label{fig4}
\end{figure}

We observe that after all splittings of $S$ we obtain a set of disjoint circles.
 
The {\em Kauffman bracket} $\langle D\rangle$ is a Laurent polynomial in a variable $A$ obtained by the following sum

$$\langle D\rangle=\sum\limits_{\tiny S \text{ state}}A^{\alpha(S)-\beta(S)} (-A^2-A^{-2})^{\gamma(S)-1}$$

where  $\alpha(S)$ and $\beta(S)$ denote the number of splittings in state $S$ of type $L_\infty$ and $L_0$ respectively and $\gamma(S)$ is the number of circles obtained after all splittings in $S$.

It is known that the Jones polynomial of a link $L$ can be calculated from the Kauffman bracket as follows,

$$V_L(t)=(-t)^{-3 w(D)/4} \langle D\rangle_{A=t^{-1/4}}$$

where $w(D)$ is the {\em writhe} of a diagram $D$ of $L$.

\begin{remark}\label{rem:mono} If we take $A=e^{\pi i/4}$ then $t=e^{-\pi i}=-1$ and, therefore 
$$\dett(L)=|V_L(-1)|=| \langle D\rangle_{A=e^{\pi i/4}}|.$$
\end{remark}

 A state $S$ is called {\em monocyclic} if $\gamma(S)=1$, that is, we obtain just one circle after all splittings in $S$.  We will focus our attention to monocyclic states since they are the only non-zero terms in the sum for $\langle D \rangle$. Indeed, since $-A^2-A^{-2}$ evaluates to $-e^{\pi i/2}-e^{-\pi i/2}=-i-(i^{-1})=-i-(-i)=0$ then the factor $(-A^2-A^{-2})^{\gamma(S)}$ in $\langle D\rangle$ is one if $\gamma(S)=1$ and zero otherwise. 


Let $D$ be a connected link diagram.  We notice that splitting a positive (resp. a negative) crossing in $D$ translates into a deletion or contraction (resp. contraction or deletion) of the corresponding edge in the associated Tait graph $G$ of $D$, Figure \ref{fig5} illustrates these choices.

\begin{figure}[H]
    \centering
    \includegraphics[width=0.5\textwidth]{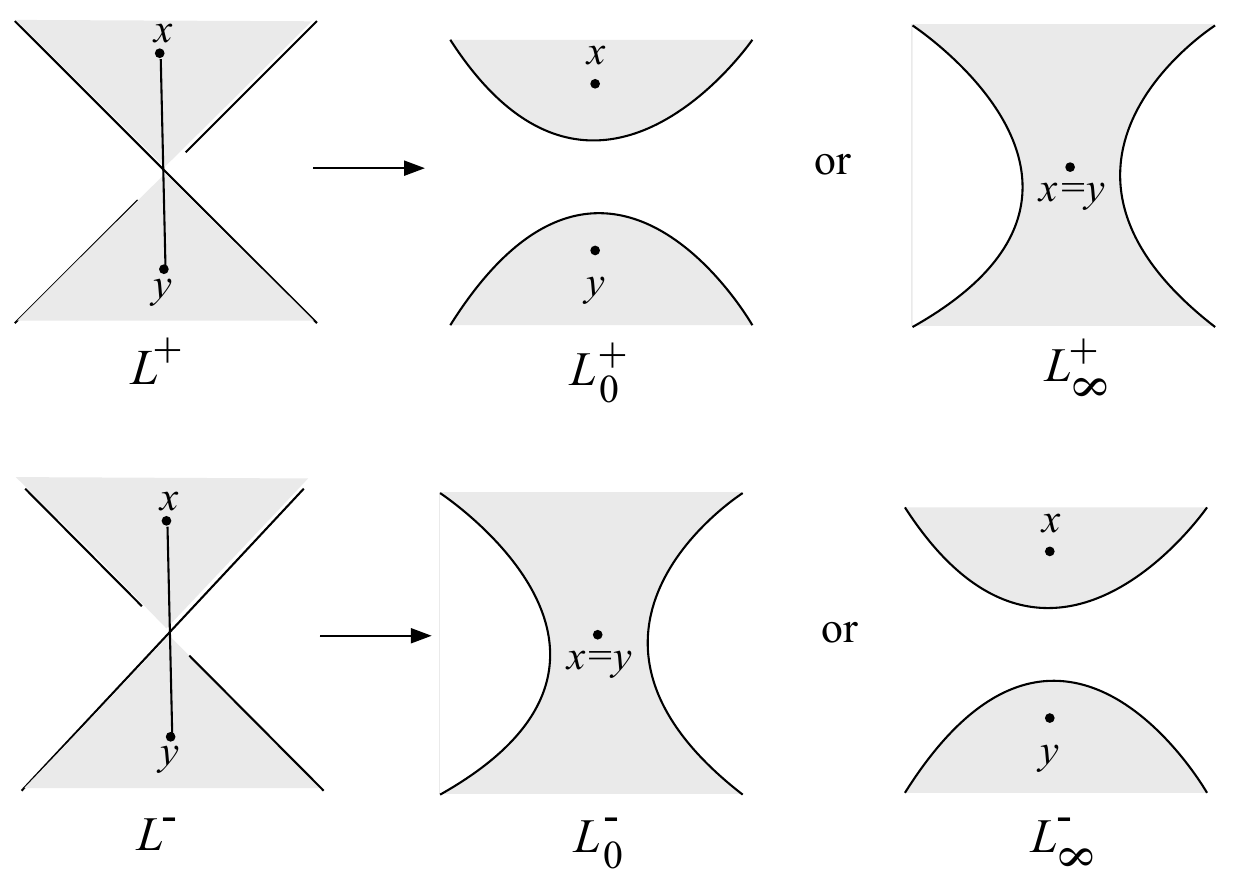}
    \caption{Splittings of a positive crossing (top) and of a negative crossing (bottom)}
    \label{fig5}
\end{figure}

We thus have that after all splittings of a state $S$ any edge in $G$ will be either {\em positive-deleted} ($L_0^+$),  {\em positive-contracted} ($L_\infty^+$), {\em negative-contracted} ($L_0^-$) or {\em negative-deleted} ($L_\infty^-$).

We recall that a {\em spanning tree} $T$ of a connected graph $G=(V,E)$ with $V(G)$ vertices and $E(G)$ edges is a tree covering the set $V(G)$, that is, $|V(T)|=|V(G)|$.

We give a characterization of the monocyclic states of a diagram.

\begin{lemma}\label{lem:spanning} Let $(G,\chi_E)$ be an edge-signed planar graph and let $D$ be the connected diagram arising from $(G,\chi_E)$. Then, a state $S$ of $D$ is monocyclic if and only if the set of (positive- and negative-) contracted edges (induced by $S$) forms a spanning tree of $G$.
\end{lemma}

\begin{proof} Suppose that $S$ is a monocyclic state. Let $X$ be the set of contracted edges induced by $S$. The signs of the edges do not matter anymore, since we are only interested in the topology of the shadow of the diagram. We first claim that $X$ is a forest. Suppose, by contradiction, that $X$ is not a forest, then $X$ admits a cycle $C$. Notice that $C$ could be a either a loop (Figure \ref{fig6} (a)), a face of $G$ (Figure \ref{fig6} (b)) or a region containing edges, in its interior, that induce a forest (see Figure \ref{fig6} (c). In any of these cases, after all splittings in $S$, the contractions of the edges of $C$ create a circle that will be disjoint from any other created circles implying the existence of at least two circles, which is a contradiction.

\begin{figure}[H]
    \centering
    \includegraphics[width=1\textwidth]{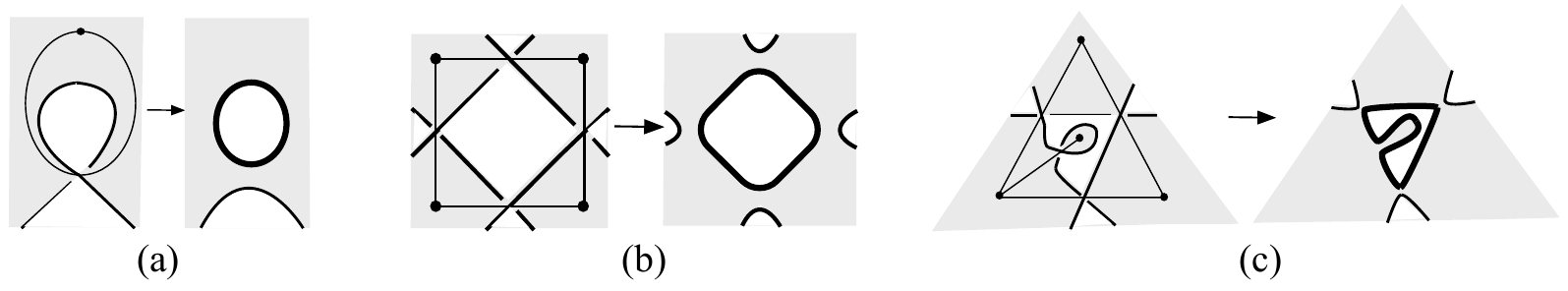}
    \caption{Negative- and positive-contractions of the edges of a cycle in $G$ and the corresponding circle (in bold) after the splittings of its edges.}
    \label{fig6}
\end{figure}

We now claim that the forest $X$ is connected. Suppose that $X$ admits at least two connected components and let $A$ be one of these components. Then, all the edges $\{x,y\}$ with $x\in V(A)$ and $y\in V(G)\setminus A$ are either negative- or positive-deleted. After all splittings in $S$ these edges will separate the circles created by the edges in $A$ from the circles created with the edges in the other components of $X$, implying again the existence of at least two circles, which is again a contradiction, see Figure \ref{fig7}. Therefore, $X$ is a tree.

\begin{figure}[H]
    \centering
    \includegraphics[width=0.5\textwidth]{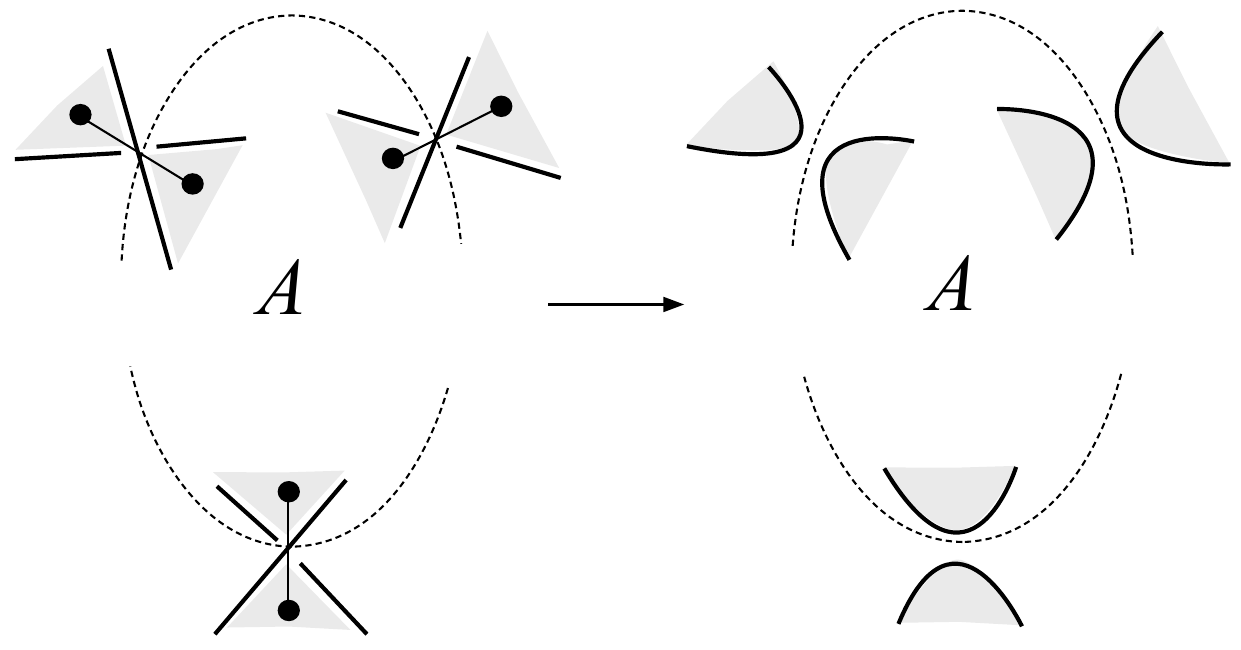}
    \caption{Negative- and positive deleted edges around the connected component $A$.}
    \label{fig7}
\end{figure}

Since $D$ is connected then $G$ cannot have isolated vertices thus $X$ covers all the vertices of $G$. Therefore, $X$ is a spanning tree of $G$.
\medskip

Suppose now that $X$ is spanning tree $T$ of $G$ formed by the negative-contracted and positive-contracted edges. We clearly have that all these edges induce one circle (wrapping up $T$), see Figure \ref{fig8}.

\begin{figure}[H]
    \centering
    \includegraphics[width=0.4\textwidth]{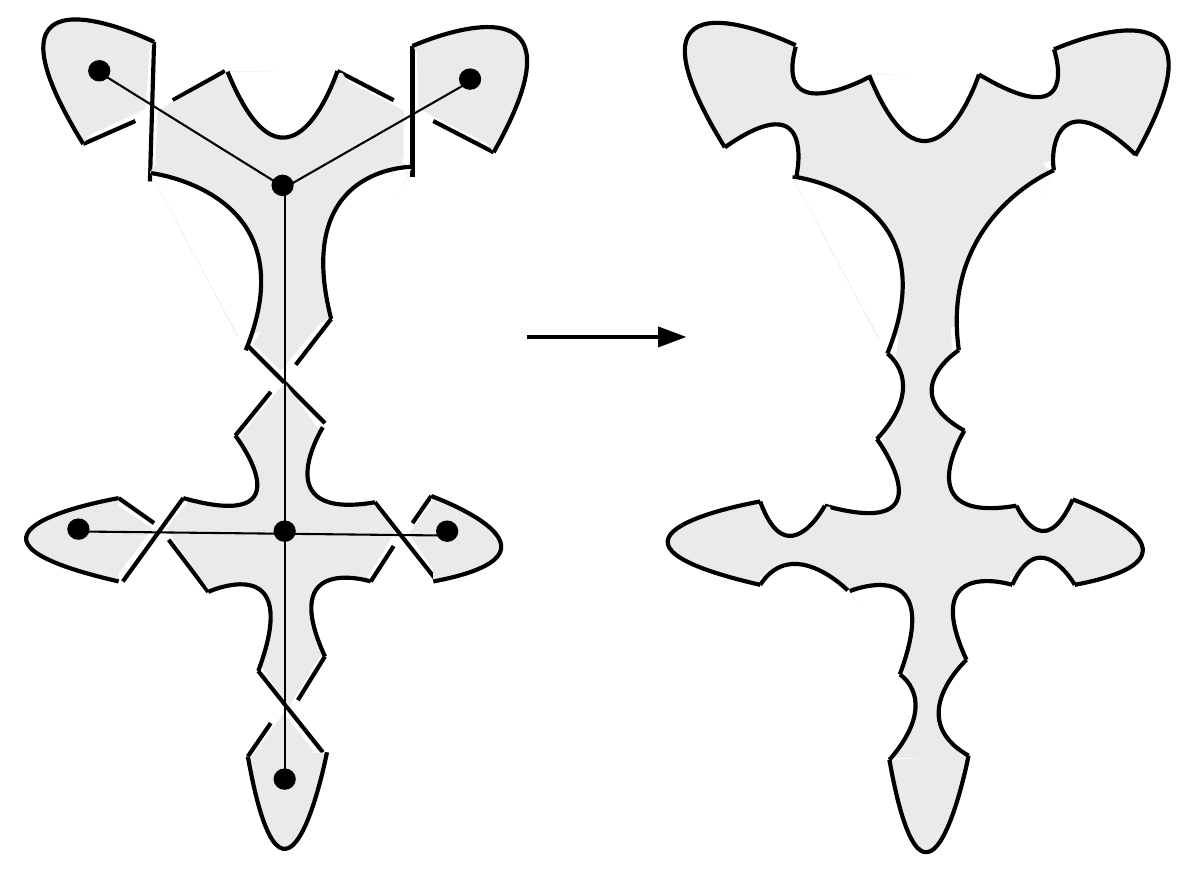}
    \caption{Negative- and positive contracted edges around a tree.}
    \label{fig8}
\end{figure}

Now, any edge $e\not\in T$ is a negative- or positive-deleted edge. Since $T$ is spanning then $T\cup e$ has a unique cycle. It is clear that the splitting corresponding to edge $e$ expand the circle induced by $T$ and so the number of circles is not increased, see Figure \ref{fig9}.

\begin{figure}[H]
    \centering
    \includegraphics[width=0.46\textwidth]{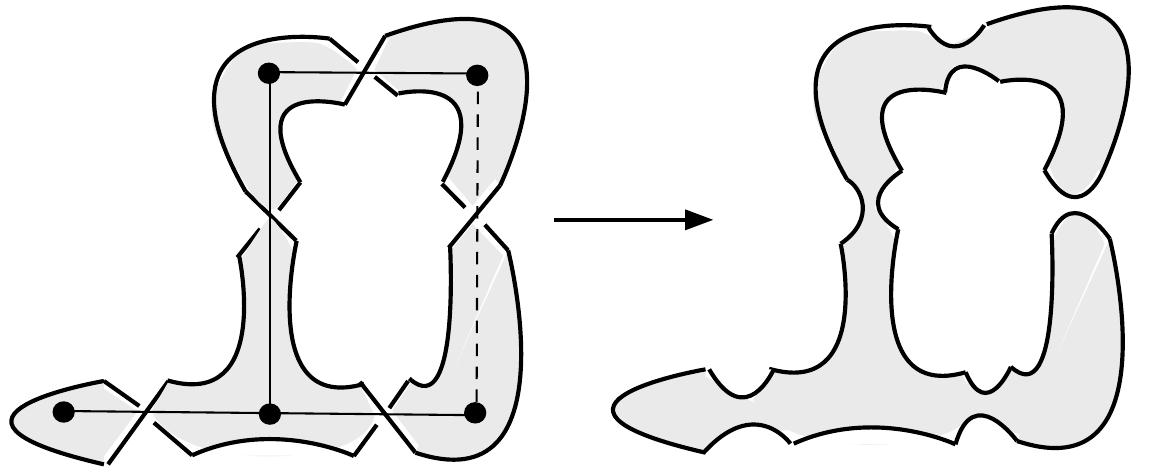}
    \caption{Splitting of negative and positive-edge of a tree (straight lines) and of a negative-deleted edge (dashed line).}
    \label{fig9}
\end{figure}

Therefore, after all splittings we end with only one circle, implying that $S$ is monocyclic.
\end{proof}


We may now prove Lemma \ref{lem:main}.
\medskip

{\em Proof of Lemma \ref{lem:main}.} Let $(G,\chi_E)$ be an edge-signed planar graph $G$ and let $S$ be a state of the diagram $D(G,\chi_E)$ on $n$ crossings. 
\smallskip

Let $c^+(S), c^-(S), d^+(S)$ and $d^-(S)$ be the number of positive-contracted, negative-contracted, positive-deleted and negative-deleted edges of $G$ respectively after all splittings $S$.

\begin{remark}\label{rem:tree} Let $S$ be a monocyclic state. Then, by Lemma \ref{lem:spanning}, $c^+(S)$ (resp. $c^-(S)$) is the number of positive  (resp. negative) edges of the corresponding spanning tree $T$ of $G$. We may write $c^+(T)$ (resp. $c^-(T)$) in this case.
\end{remark}

Let $n=n^++n^-$ where $n^+$ and $n^-$ denote the number of positive and negative crossings in $D$ respectively.
\smallskip

We have that 
\begin{equation}\label{eq:1}
\alpha(S)=c^+(S)+d^-(S), \ \ \   \beta(S)=c^-(S)+d^+(S)
\end{equation}

and

\begin{equation}\label{eq:1a}
c^+(S)+d^+(S)=n^+.
\end{equation}

Moreover, by Lemma \ref{lem:spanning}, we have

\begin{equation}\label{eq:2}
c^+(S)+c^-(S)=n-1.
\end{equation}

By combining \eqref{eq:1a} and \eqref{eq:2} we obtain

\begin{equation}\label{eq:3}
d^+(S)=c^-(S)+n^+-n+1=c^-(S)-n^-+1.
\end{equation}

We have
$$\begin{array}{lll}
\langle D\rangle_{A=e^{\pi i/4}} & =\sum\limits_{\tiny S \text{-monocyclic}}A^{\alpha(S)-\beta(S)} &\\ &=\sum\limits_{\tiny S \text{-monocyclic}}A^{n-2\beta(S)}&\\
&=\sum\limits_{\tiny S \text{-monocyclic}}A^{n-2c^-(S)-2d^+(S)} &\text{ (by \eqref{eq:1})} \\
&=\sum\limits_{\tiny S \text{-monocyclic}}A^{-4c^-(S) +2n^-+n-2} & \text{ (by \eqref{eq:3})} \\
&=A^{2n^-+n-2}\sum\limits_{\tiny T \text{- spanning tree}}A^{-4c^-(T)} & \text{ (by Remark \ref{rem:tree})} \\
&=e^{\pi i \frac{(2n^-+n-2)}{4}}\sum\limits_{\tiny T \text{- spanning tree}}e^{-\pi i c^-(T)} & \\
&=e^{\pi i \frac{(2n^-+n-2)}{4}}\sum\limits_{\tiny T \text{- spanning tree}}(-1)^{c^-(T)} & \\
\end{array}$$
\smallskip

Finally, since $c^-(T)$ is even (resp. odd) if and only if  $T$ is positive (resp. negative) then, by Remark \ref{rem:mono}, we obtain
$$\begin{array}{ll}
\dett(L) & =| \langle D\rangle_{A=e^{\pi i/4}}| =\left| \sum\limits_{\tiny T \text{- spanning tree}}(-1)^{c^-(T)} \right|\\
\\
& = |\# \{\text{positive spanning trees of } G\}-\# \{\text{negative spanning trees of } G\}|.
\end{array}$$
\hfill$\square$

\end{document}